\documentclass[10pt]{article}
\usepackage{bbm}
\usepackage{mathrsfs}
\usepackage{amsfonts}
\usepackage{amsthm,amsmath,amssymb,anysize}
\newtheorem{lemma}{Lemma}[section]
\newtheorem{theorem}[lemma]{Theorem}

\newtheorem{definition}[lemma]{Definition}

\usepackage[numbers,sort&compress]{natbib}
\marginsize{3.2cm}{3.2cm}{3.6cm}{3cm}
\numberwithin{equation}{section}

\title{\textsf{ Cohomology of $\frak{sl}(2)$ and its applications in prime characteristic \footnote{supported by Heilongjiang Provincial Natural Science Foundation of China (YQ2020A005) and the Natural Science Foundation  of China (12061029).}}}
\author{\textsc{Shujuan Wang$^{1}$ and
  \textsc{Zhaoxin Li$^{2,}$}}\footnote{Correspondence: LZX15765737480@163.com}\\
  {\small \textit{$^1$Department of Mathematics, Shanghai Maritime University,}}\\
\small \textit{Shanghai 201306, China} \\
 \small\textit{$^2$School of Mathematical Sciences, Heilongjiang University,}\\ \small\textit{Harbin 150080, China} }
\date{ }
\begin{document}
\maketitle
\begin{quotation}
\small\noindent \textbf{Abstract}:
Over a field of characteristic $p>2,$
 the first cohomology of the 3-dimensional simple Lie algebra $\frak{sl}(2)$ with
coefficients in all simple modules  is determined, which implies  Whitehead's first lemma
is not true in prime characteristic.
As applications, derivations and 2-local derivations of $\frak{sl}(2)$ on  any simple module are also be characterized.

 \vspace{0.2cm} \noindent{\textbf{Keywords}}: $\frak{sl}(2)$,  simple modules, the first cohomology, (2-local) derivations.

\vspace{0.1cm} \noindent \textbf{Mathematics Subject Classification
}: 17B40, 17B50, 17B56

\end{quotation}
\setcounter{section}{0}
\section{Introduction}
 Cohomology is an important tool in modern mathematics and theoretical physics.
Recall, for instance, that relative cohomology is fundamental in the Borel-Weil-Bott theory (see \cite{Penkov}),
or that cohomology of nilpotent radicals of parabolic subalgebras is crucial in the Kazhdan-Lusztig theory (see \cite{Brundan}).
Particularly, the first cohomology is closely related to extensions of modules.
Let $\frak{g}$ be a finite-dimensional simple Lie algebra
and $M$ be any finite-dimensional simple $\frak{g}$-module over a field of characteristic 0.
It is Whitehead's first lemma that tells us the first cohomology $\mathrm{H}^1(\frak{g}, M) = 0$
(see \cite[Theorem 7.8.9 or Corollary 7.8.10]{cohomology}).
In this paper, the underlying field $\mathbb{F}$ is assumed to be of characteristic $p>2$.
We determine the first cohomology of the 3-dimensional simple Lie algebra $\frak{sl}(2)$
with coefficients in all simple modules over $\mathbb{F}$ (see Theorem \ref{220205} in Section \ref{1127}),
which implies that Whitehead's first lemma is not true over a field of prime characteristic.
As an application, we also characterize all derivations of $\frak{sl}(2)$ on any simple module over $\mathbb{F}$
(see Theorem \ref{2201091408} in Subsection \ref{1710}).

Motivated by Bre$\check{s}$ar  and Zhao's work on skewsymmetric  biderivations,
generalizing the concept of skewsymmetric biderivations
of Lie algebras from the adjoint module to any module in the paper \cite{zhao},
we  generalize the notion  of 2-local derivations
of a finite-dimensional Lie algebra from the adjoint module to any finite-dimensional module over $\mathbb{F}$,
and characterize all 2-local derivations  of $\frak{sl}(2)$ on  any  simple module (see Theorem \ref{1711} in Subsection \ref{1712}).
In fact, the concept of 2-local derivations  of Lie algebras (on the adjoint module)
is first introduced in 2015 (see \cite{20151,20152}),
which have aroused the interest
of a great many authors in recent years, see \cite{1,2,3} for example.
It is well-known that derivations and 2-local derivations of algebras are influential
and far-reaching, which are very important subjects in the study of both algebras and
their generalizations.
\subsection{The first cohomology of $\frak{g}$ with coefficients in $\frak{g}$-modules}
Let $\frak{g}$ be a finite-dimensional Lie algebra and $M$ a finite-dimensional $\frak{g}$-module in this subsection.
We call a linear map
$\phi: \frak{g}\longrightarrow M$ to be a derivation
of $\frak{g}$ on $M$, if
\begin{eqnarray*}
\phi([x,y])=x\phi(y)-y\phi(x),\forall x,y\in \frak{g}.
\end{eqnarray*}
Denote by $\mathrm{Der}(\frak{g}, M)$ the space consisting of
derivations of $\frak{g}$ on $M$.
For $m\in M$, we call a map $D: \frak{g}\longrightarrow M$
to be an inner derivation determined by $m$, denoted by $D_m$,
if
$$D(x)=x m,\quad \forall x\in \frak{g}.$$
Denote by $\mathrm{Ider}(\frak{g}, M)$ the space consisting of inner derivations of $\frak{g}$ on $M$.
Suppose that
$\mathfrak{h}$ is a Cartan subalgebra of $\frak{g}$ and
$\frak{g}$ and $M$ admit the weight space decompositions with respect to $\mathfrak{h}$:
$$\frak{g}=\oplus_{\alpha\in\mathfrak{h}^{\ast}}\frak{g}_{\alpha},\quad M=\oplus_{\alpha\in\mathfrak{h}^{\ast}}M_{\alpha}.$$
We call a linear map $\phi: \frak{g}\longrightarrow M$ to be a weight-map,
if $\phi(\frak{g}_{\alpha})\subseteq M_{\alpha}, \forall\alpha\in\mathfrak{h}^{\ast}.$
Furthermore, if a weight-map $\phi\in \mathrm{Der}(\frak{g}, M)$,
then $\phi$ is said to be a weight-derivation.
Denote by $\mathrm{Wder}(\frak{g}, M)$ the space consisting of weight-derivations of $\frak{g}$ on $M$.
From \cite[Theorem 1.1]{RF}, the first cohomology of $\frak{g}$  with the
coefficient in the $\frak{g}$-module $M$
\begin{align}\label{1022}
\mathrm{H}^1(\frak{g},M)&=\mathrm{Der}(\frak{g},M)/\mathrm{Ider}(\frak{g},M)\nonumber\\
&=\left(\mathrm{Wder}(\frak{g},M)+\mathrm{Ider}(\frak{g},M)\right)/\mathrm{Ider}(\frak{g},M).
\end{align}

\subsection{Simple modules of $\frak{sl}(2)$}\label{1113}

For any $k\in \{0,2,-2\}$ and $a\in \mathbb{F}_{p},$
let $\Phi(a)$ and $\Phi(a, k)$ be the integer numbers
such that both $\Phi(a)$ and $\Phi(a, k)$ are in $\{0, 1, \cdots, p-1\}$,
and
$$\Phi(a)\equiv a(\mathrm{mod}p),\quad\Phi(a, k)\equiv\frac{a+k}{2}(\mathrm{mod}p).$$
We will follow the reference \cite[Sec. 5.2]{modularliealg} for the representation theory
of  the 3-dimensional simple Lie algebra $\frak{sl}(2)$.
For the convenience of the readers, we summarize some information.
Recall that $\mathfrak{sl}(2)$
has a standard basis:
$$
	h=\begin{pmatrix}
		1& 0\\
		0& -1
	\end{pmatrix},\quad
    e=\begin{pmatrix}
		0& 1\\
		0& 0
	\end{pmatrix},\quad
	f=\begin{pmatrix}
		0& 0\\
		1& 0
	\end{pmatrix}.
$$
$\frak{sl}(2)$ is a restricted Lie algebra with the only $p$-map such that
$$h^{[p]}=h,\quad e^{[p]}=f^{[p]}=0.$$
Denote by $U(\frak{sl}(2))$ the universal enveloping algebra of $\frak{sl}(2)$.
Note that $x^p-x^{[p]}$ lies in the center of $U(\frak{sl}(2))$ for all $x\in \frak{sl}(2)$.
Let $V$ be a simple $\frak{sl}(2)$-module.
Schur's Lemma tells us there exists $\chi\in \frak{sl}(2)^*$,
such that
$$x^pv-x^{[p]}v=\chi(x)^pv, \quad  x\in \frak{sl}(2),\;\; v\in V.$$
In this case we also call $V$  a simple $\frak{sl}(2)$-module with  $p$-character $\chi$.
Fix $\chi\in \frak{sl}(2)^*$.
Denote by $I_{\chi}$  the two-sided ideal of $U(\frak{sl}(2))$ generated by the elements
$x^p-x^{[p]}-\chi(x)^p$
for all $x\in \frak{sl}(2)$.
Write $U_{\chi}(\frak{sl}(2))=U(\frak{sl}(2))/I_{\chi}$,
which is called the reduced enveloping algebra with $p$-character $\chi$.
Note that a simple $\frak{sl}(2)$-module with $p$-character $\chi$
is the same as a simple $U_{\chi}(\frak{sl}(2))$-module.
Any $p$-character $\chi'$ is $\mathrm{SL}(2)$-conjugate to a $p$-character $\chi$ with $\chi(e)=0$
and $U_{\chi'}(\frak{sl}(2))=U_{\chi}(\frak{sl}(2))$.
Therefore the study of simple  $\frak{sl}(2)$-modules
is reduced to a problem of studying simple ones  with  $p$-character $\chi$  when $\chi$ runs over the
representatives of coadjoint $\mathrm{SL}(2)$-orbits in $\frak{sl}(2)^*$.
In this paper we make a convention that the symbol $\chi$ implies $\chi\in \frak{sl}(2)^*$ and $\chi(e)=0$.

Below we recall simple $U_{\chi}(\frak{sl}(2))$-modules constructed in  \cite[p. 207]{modularliealg}.
 Write
 $$\Lambda_{\chi}=\left\{\lambda\in \mathbb{F}\mid \lambda^p-\lambda=\chi(h)^p\right\}.$$
Let $\mathbb{F}v$ be a 1-dimensional vector space and
$hv=\lambda v$.
Then $\mathbb{F}v$ is a $U_{\chi}(\mathbb{F}h)$-module
 if and only if $\lambda\in \Lambda_{\chi}$.
By letting $ev=0$,  $\mathbb{F}v$ can be extended to a 1-dimensional $U_{\chi}(\mathbb{F}h+\mathbb{F}e)$-module, 
which is because $\chi(e)=0$. 
Write
$$ Z_{\chi}(\lambda):=U_{\chi}(\frak{sl}(2))\otimes_{U_{\chi}(\mathbb{F}h+\mathbb{F}e)}\mathbb{F}v,\quad \lambda\in \Lambda_{\chi},$$
which is called the Verma module of $U_{\chi}(\frak{sl}(2))$  with   highest weight $\lambda$ and $p$-character $\chi$.
Then $Z_{\chi}(\lambda)$ has a basis
$$v_0:=1\otimes v,\quad v_{i}:=f^i\otimes v,\;\; i=1,\;2,\;\ldots,\; p-1.$$
Denote by $L_{\chi}(\lambda)$ the unique simple quotient of $ Z_{\chi}(\lambda)$,
which is also of highest weight $\lambda$ and $p$-character $\chi$.
If $v\in Z_{\chi}(\lambda)$, then we also use $v$
to represent the residue class containing $v$ in $L_{\chi}(\lambda)$.
We describe the structure of $L_{\chi}(\lambda)$  in the following:
\begin{itemize}
\item[(1)]
If $\chi(h)\neq0$ or $\chi(f)\neq0$, then $L_{\chi}(\lambda)$
has a standard basis: $v_{0},v_{1},\ldots,v_{p-1}$, with the $\frak{sl}(2)$-action as follows:
   \begin{align*}
   	   ev_i=i(\lambda-i+1)v_{i-1}, \quad hv_i=(\lambda-2i)v_i, \quad fv_i=\left\{\begin{matrix}
   		v_{i+1},&i<p-1\\
   		\chi (f)^pv_0,&i=p-1.
   	\end{matrix}\right.
   \end{align*}
\item[(2)]
If $\chi=0$, then $L_{\chi}(\lambda)$
has a standard basis: $v_{0},v_{1},\ldots,v_{\Phi(\lambda)}$, with the $\frak{sl}(2)$-action as follows:
\begin{align*}
     	ev_i=i(\lambda-i+1)v_{i-1}, \quad hv_i=(\lambda-2i)v_i, \quad fv_i=\left\{\begin{matrix}
     		v_{i+1},&i<\Phi(\lambda)\\
     		0,&i=\Phi(\lambda).
     	\end{matrix}\right.
     \end{align*}
\end{itemize}

\section{$\mathrm{H}^1(\frak{sl}(2), L_{\chi}(\lambda))$}\label{1127}
The following two theorems determine the first cohomology
of $\frak{sl}(2)$ with coefficients in $L_{\chi}(\lambda)$,
which implies that Whitehead's first lemma in \cite[Corollary 7.8.10]{cohomology}
is not true over a field of prime characteristic.
\begin{theorem}\label{220205}
		$$\mathrm{H}^1(\mathfrak{sl}(2), L_{\chi}(\lambda))=\left\{\begin{matrix}
			\mathbb{F}d_{1}+\mathbb{F}d_{2},& (\chi, \lambda)=(0, p-2)  \\
			0,& \mbox{otherwise},
		\end{matrix}\right. $$
where both $d_1$ and $d_2$ are linear maps from $\frak{sl}(2)$ to $L_{0}(p-2)$ such that
\begin{align*}
d_{1}(h)= d_{1}(f)&=0,\quad   d_{1}(e)=v_{p-2},\\
d_{2}(h)= d_{2}(e)&=0,\quad d_{2}(f)=v_{0}.
		\end{align*}
	\end{theorem}
\begin{proof}
We shall prove this conclusion by steps.

\textit{Step 1: Let $\chi(h)\neq 0$.}
It follows that $\lambda\notin \mathbb{F}_p$ from $\lambda\in\Lambda_{\chi}$.
Then $L_{\chi}(\lambda)_{\alpha}=0$ for any $\alpha=0,2,-2$,
which implies that $\mathrm{Wder}(\frak{sl}(2), L_{\chi}(\lambda))=0.$
As a result, by Equation (\ref{1022}), we get $\mathrm{H}^1(\frak{sl}(2),L_{\chi}(\lambda))=0$ under this case.

\textit{Step 2: Let $\chi(h)= 0$.}
Since $hv_i=(\lambda-2i)v_{i}$,
then we let $\phi\in \mathrm{Wder}(\frak{sl}(2), L_{\chi}(\lambda))$
such that
$$\phi(h)=av_{\Phi (\lambda ,0)},\quad \phi(e)=bv_{\Phi (\lambda ,-2)},\quad \phi(f)=cv_{\Phi (\lambda ,2)}.$$
For any $\alpha \in (\mathbb{F}h)^{*}$ and $x\in \frak{sl}(2)_{\alpha},$
by $$\phi([h,x])=\phi(\alpha(h)x)=\alpha(h)\phi(x)$$ and
$$h\phi(x)-x\phi(h)=\alpha(h)\phi(x)-x\phi(h),$$ we get $x\phi(h)=0$.
Considering that  $x$ are arbitray and $L_{\chi}(\lambda)$ is a simple $\frak{sl}(2)$-module,
we get $\phi(h)=0$.

\textit{Step 2.1: Let $\chi=0$.}
In this case, $L_0(\lambda)$ has a basis $v_0, v_1, \ldots, v_{\Phi(\lambda)}$.
Write $k\in\{0,2,-2\}$ and $\mathrm{I}=\{0,1,\ldots, \Phi(\lambda)\}$.
Then
$$v_{\Phi(\lambda,k)}\neq 0\Longleftrightarrow \Phi(\lambda,k)\in \mathrm{I}.$$
In fact, for any $k\in \{0,2,-2\}$,
we get
$$\Phi(p-1, k)\in \mathrm{I}, \quad \Phi(2, k)\in \mathrm{I}, \quad \Phi(1, k)\notin \mathrm{I}(p\geq 5),\quad \Phi(p-3, k)\in \mathrm{I}(p\geq 7)$$  and the following conclusions:
\begin{itemize}
\item[(a)]If $p=3$, then
$$
\Phi(0, k)\in \mathrm{I}\Longleftrightarrow k=0;\quad \Phi(1, k)\in \mathrm{I}\Longleftrightarrow k=\pm2.
$$
\item[(b)]If $p\geq5$, then
$$\Phi(0, k)\in \mathrm{I}\Longleftrightarrow k=0; \quad
\Phi(p-2, k)\in \mathrm{I}\Longleftrightarrow k=\pm2.$$
\item[(c)]If $p\geq7$ and $3\leq\Phi(\lambda)<p-3$, then
$$\Phi(\lambda, k)\in\mathrm{I}\Longleftrightarrow \Phi(\lambda)\mbox{\; is even}.$$
\end{itemize}
In the following, we shall determine $\phi$.
\begin{itemize}
\item If $\lambda=0$, then $\Phi(\lambda, k)\notin \mathrm{I}$ for $k=\pm2$.
The fact $\phi(e)=\phi(f)=0$ implies $\phi([e,f])=0$,
that is, $\phi(h)=0$.
It follows that $\phi=0.$
As a result, by Equation (\ref{1022}),
we get $\mathrm{H}^1(\frak{sl}(2),L_{0}(\lambda))=0$ under this case.
	
\item If $\lambda=1$ with $p\geq 5$; or $3\leq \Phi(\lambda)< p-3$ is odd with $p\geq 7$,
then $\Phi(\lambda, k)\notin \mathrm{I}$ for $k=0,\pm2$. It follows that $\phi=0.$
As a result, by Equation (\ref{1022}), we get $\mathrm{H}^1(\frak{sl}(2),L_{0}(\lambda))=0$ under this case.

\item If $\lambda=p-2$,
then the fact  $\phi(h)=0$  implies $\phi=bd_{1}+cd_{2}$.
By a direct computation, we get that both $d_{1}$ and $d_{2}$ are lie in $\mathrm{Wder}(\frak{sl}(2), L_{0}(p-2))$.
Claim that $bd_1+cd_2$ is inner if and only if $b=c=0$.
In fact, if $bd_1+cd_2=D_v$ for $v=\sum^{p-2}_{i=0}a_iv_i\in L_{0}(p-2)$,
then
$$
0=(bd_1+cd_2)(h)=D_v(h)=h\left(\sum^{p-2}_{i=0}a_iv_i\right)=\sum^{p-2}_{i=0}a_i(-2-2i)v_i.
$$
It follows that $a_i=0$ from $-2-2i\neq 0$ for $i=0,1,\ldots, p-2$,
which forces $v=0$, and then $b=c=0$.
As a result, by Equation (\ref{1022}), we get
$$\mathrm{H}^1(\frak{sl}(2),L_{0}(p-2))=\mathbb{F}d_1+\mathbb{F}d_2.$$

\item If $ \lambda=2$; or $\lambda\in\{p-1,p-3\}$ with $p\geq 5$; or $3\leq \Phi (\lambda )< p-3$ is even with $p\geq7$,
then $\Phi(\lambda, k)\in \mathrm{I}$ for $k=0,\pm2$.
The fact $\phi(h)=0$ implies $\phi([e,f])=0$.
From
		$$ e\phi(f)-f\phi(e)=cev_{\Phi(\lambda,2)}-bfv_{\Phi(\lambda,-2)}=c\frac{\lambda }{2}\left(\frac{\lambda }{2}+1\right)v_{\Phi(\lambda ,0 )}-bv_{\Phi(\lambda ,0 )},$$
it follows that $ b=\frac{\lambda }{2}\left(\frac{\lambda }{2}+1\right)c, $
which implies that $ \phi=cD_{v_{\Phi (\lambda ,0)}}. $
As a result, by Equation (\ref{1022}), we get $\mathrm{H}^1(\frak{sl}(2),L_{0}(\lambda))=0$ under this case.
\end{itemize}

\textit{Step 2.2: Let  $\chi(f)\neq 0$.}
In this case, $L_{\chi}(\lambda)$ has a basis $v_0, v_1, \ldots, v_{p-1}$.
The fact $\phi(h)=0$ implies $\phi([e,f])=0$.
That is,
		\begin{align*}
			0=e\phi(f)-f\phi(e)&=cev_{\Phi(\lambda,2)}-bfv_{\Phi(\lambda,-2)}\\
			&=\left\{\begin{matrix}
   		\left(-b+c\frac{\lambda}2\left(\frac{\lambda}{2}+1\right)\right)v_{\Phi(\lambda,0)},&\lambda\neq 0\\
   		-b\chi{(f)}^pv_0,&\lambda=0.
   	\end{matrix}\right.
		\end{align*}
Hence
$b=\frac{\lambda}{2}\left(\frac{\lambda}{2}+1\right)c$,
that is $\phi=D_{v_{\Phi{(\lambda,0)}}}$.
As a result, by Equation (\ref{1022}), we get $\mathrm{H}^1(\frak{sl}(2),L_{0}(\lambda))=0$ under this case.
\end{proof}

\section{Applications}
In this section, we shall determine all derivations and 2-local derivations of $\frak{sl}(2)$ on any $L_{\chi}(\lambda)$.

\subsection{Derivations}\label{1710}
By Equation (\ref{1022}) and Theorem \ref{220205}, we have the following theorem,
which determines all derivations of $\frak{sl}(2)$ on any $L_{\chi}(\lambda)$.
\begin{theorem}\label{2201091408}
$$\mathrm{Der}(\frak{sl}(2), L_{\chi}(\lambda))=\left\{\begin{matrix}
   		\mathrm{Ider}(\frak{sl}(2), L_{\chi}(\lambda))+\mathbb{F}d_1+\mathbb{F}d_2,&(\chi, \lambda)=(0, p-2)\\
   		\mathrm{Ider}(\frak{sl}(2), L_{\chi}(\lambda)),& \mbox{otherwise}.
   	\end{matrix}\right.$$

\end{theorem}

\subsection{2-local derivations}\label{1712}

In the following, we introduce the definition
of 2-local derivations  of a finite-dimensional Lie algebra on any finite-dimensional module,
which generalizes the usual ones from the adjoint module to any finite-dimensional module.
\begin{definition}
Let $\frak{g}$ be a finite-dimensional Lie algebra, $M$ a finite-dimensional $\frak{g}$-module
and $\varphi$   a map from $\frak{g}$ to $M$.
If for any $x, y\in \frak{g}$, there exists $D_{x,y}\in \mathrm {Der} (\frak{g},M)$ {\rm(}depending on $x, y${\rm)},
such that
$$\varphi(x)=D_{x,y}(x),\quad \varphi (y)=D_{x,y}(y),$$
then $\varphi$ is said to be a 2-local derivation of $\frak{g}$ on $M$.
\end{definition}

Below we give a technical lemma which will be used in the future.

\begin{lemma}\label{202202101608}
Let $\frak{g}$ be a finite-dimensional  Lie algebra, $M$ a nontrivial and simple $\frak{g}$-module of finite dimension. Then
		$$
		a_{1}D_{u_{1}}+a_{2}D_{u_{2}}+\cdots +a_{k}D_{u_{k}}=0
		\Longleftrightarrow a_{1}u_{1}+a_{2}u_{2}+\cdots +a_{k}u_{k}=0,
		$$
		where $u_{1},u_{2},\ldots,u_{k} \in M, a_{1},a_{2},\ldots,a_{k} \in  \mathbb{F}$.
		Particularly,  	
		\begin{center}
			$u_{1},\ldots,u_{k}$ is a basis of $M$ if and only if $D_{u_{1}},\ldots ,D_{u_{k}}$ is a basis of $\mathrm{Ider}(\frak{g},M)$.
		\end{center}

\end{lemma}
\begin{proof}
	Let $ x\in \frak{g}$. Then
	\begin{align*}
	&(a_{1}D_{u_{1}}+a_{2}D_{u_{2}}+\cdots +a_{k}D_{u_{k}})(x)\\
    &=a_{1}D_{u_{1}}(x)+a_{2}D_{u_{2}}(x)+\cdots +a_{k}D_{u_{k}}(x)\\
	&=a_{1}xu_{1}+a_{2}xu_{2}+\cdots +a_{k}xu_{k}\\
	&=x(a_{1}u_{1}+a_{2}u_{2}+\cdots +a_{k}u_{k})\\
	&=D_{a_{1}u_{1}+a_{2}u_{2}+\cdots +a_{k}u_{k}}(x).
    \end{align*}
That is,
$$
	a_{1}D_{u_{1}}+a_{2}D_{u_{2}}+\cdots +a_{k}D_{u_{k}}=D_{a_{1}u_{1}+a_{2}u_{2}+\cdots +a_{k}u_{k}},
 $$
which implies the ``if" direction.

Let $a_{1}D_{u_{1}}+a_{2}D_{u_{2}}+\cdots +a_{k}D_{u_{k}}=0$.
For any $x\in \frak{g}$, we have
  $$
  (a_{1}D_{u_{1}}+a_{2}D_{u_{2}}+\cdots +a_{k}D_{u_{k}})(x)=0,
  $$
then
  $$
 D_{a_{1}u_{1}+a_{2}u_{2}+\cdots +a_{k}u_{k}}(x) =x(a_{1}u_{1}+a_{2}u_{2}+\cdots +a_{k}u_{k}).
  $$
It follows that
  $a_{1}u_{1}+a_{2}u_{2}+\cdots +a_{k}u_{k} \in M^{\frak{g}},$
 where $ M^{\frak{g}}:=\{m\in M\mid xm=0, \forall x\in \frak{g}\}.$
Since $M$ is nontrivial and simple,  we get $ M^{\frak{g}}=0$.
 Then
  $$a_{1}u_{1}+a_{2}u_{2}+\cdots +a_{k}u_{k}=0.$$
\end{proof}

Now we are in the position to give the main result on 2-local derivations of $\frak{sl}(2)$ on any $L_{\chi}(\lambda)$.
\begin{theorem}\label{1711}
Let $\varphi$ be a 2-local derivation of $\frak{sl}(2)$ on any $L_{\chi}(\lambda)$.
\begin{itemize}
\item[(1)] If $(\chi, p,\lambda)\neq (0,3,1)$,
then $\varphi$ is a derivation.
\item[(2)] If $(\chi, p,\lambda)= (0,3,1)$,
then for any $x=x_ee+x_ff+x_hh\in \frak{sl}(2)$
with $x_e,x_f,x_h\in \mathbb{F}$, there exist a derivation $D\in\mathrm{Der}(\frak{sl}(2),L_{\chi}(\lambda))$
and two maps $m(x), n(x):  \frak{sl}(2)\longrightarrow\mathbb{F}$,
such that
$$\varphi(x)=D+\left\{\begin{matrix}
   		0,&x_ex_f=0\\
   		m(x)v_0+n(x)v_1,& x_ex_f\neq 0.
   	\end{matrix}\right.$$
\end{itemize}
\end{theorem}

\begin{proof}
For any fixed $x',y'\in \frak{sl}(2)$,
then by definition, there exists a derivation $D_{x',y'}: \mathfrak{sl}(2)\rightarrow L_{\chi}(\lambda)$,
such that $\varphi(x')=D_{x',y'}(x')$ and $\varphi(y')=D_{x',y'}(y')$.
Write $\varphi_1=\varphi-D_{x',y'}$.
Then $\varphi_1$ is also a 2-local derivation of $\mathfrak{sl}(2)$ on $L_{\chi}(\lambda)$,
and
\begin{equation}\label{202202101637}
\varphi_1(x')=\varphi_1(y')=0.
\end{equation}
In the following we shall prove by steps that
there exist maps $m(x)$ and $n(x)$ from $\frak{sl}(2)$ to $\mathbb{F}$, such that
$$\varphi_1(x)=\left\{\begin{matrix}
   		m(x)v_0+n(x)v_1,& x_ex_f\neq 0\;\mbox{and}\;(\chi, p,\lambda)= (0,3,1)\\
   		0,&\mbox{otherwise}
   	\end{matrix}\right.$$
   for any $x=x_ee+x_ff+x_hh\in \frak{sl}(2)$ with $x_e,x_f,x_h\in\mathbb{F}$.

\textit{Step 1: Let $\chi(h)\neq 0$.}
Let $x'=h$ and $y'=f$ in Equation (\ref{202202101637}).
Then for any $x \in \mathfrak{sl}(2)$,
by definition there exists a derivation $D_{x,h}$,
letting $\sum_{i=0}^{p-1}a_i(x,h)D_{v_i}$ by Theorem \ref{2201091408}, Lemma \ref{202202101608} and $\chi(h)\neq 0$, such that
		\begin{align*}
			\varphi_1(h)=\sum_{i=0}^{p-1}a_i(x,h)D_{v_i}(h)=\sum_{i=0}^{p-1}a_i(x,h)(\lambda-2i)v_i=0.
		\end{align*}
It follows that  $\lambda \notin \mathbb{F}_p$ from $\chi(h)\neq 0$.
Then each $a_i(x,h)=0$,
which implies $\varphi_1(x)=0$.

\textit{Step 2: Let $\chi(e)=\chi(h)=0$ and $\chi(f)\neq 0$.}
Let $x'=h$ and $y'=f$ in Equation (\ref{202202101637}).
Then for any $x \in \mathfrak{sl}(2)$,
by definition  there exists a derivation $D_{x,f}$,
letting $\sum_{i=0}^{p-1}a_i(x,f)D_{v_i}$ by Theorem \ref{2201091408}, Lemma \ref{202202101608} and $\chi(f)\neq 0$, such that
		\begin{align*}
			\varphi_1(f)=\sum_{i=0}^{p-1}a_i(x,f)D_{v_i}(f)=\sum_{i=0}^{p-2}a_i(x,f)v_{i+1}+a_{p-1}(x,f)\chi(f)^pv_0=0.
		\end{align*}
It follows that each $a_i(x,f)=0$,
which implies $\varphi_1(x)=0$.

\textit{Step 3: Let $\chi=0$.}
By letting $x'=e$ and $y'=f$ in Equation (\ref{202202101637}),
then we may assume that $\varphi_1(e)=\varphi_1(f)=0.$
Again by definition, for any fixed $x,y \in \mathfrak{sl}(2)$,
 there exists a derivation $D_{x,y}: \mathfrak{sl}(2)\rightarrow L_{\chi}(\lambda)$,
such that $\varphi_1(x)=D_{x,y}(x)$ and $\varphi_1(y)=D_{x,y}(y)$.
By Theorem \ref{2201091408}, Lemma \ref{202202101608} and $\chi= 0$,
we may set
		\begin{align*}
			D_{x,y}=\sum_{i=0}^{ \Phi(\lambda )}a_i(x,y)D_{v_i}+\delta _{\Phi(\lambda),p-2}\sum_{j=1}^{2}b_j(x,y)d_j,
		\end{align*}
where $d_1$ and $d_2$ are defined in Theorem \ref{220205}.
Hereafter, the symbol $\delta _{i,j}$ is the Kronecker Symbol.

\textbf{Claim 1: $\varphi_1(h)=0$.}
In fact, on one hand,
\begin{align*}
			\varphi_1(h)=D_{h,e}(h)&=\sum_{i=0}^{ \Phi(\lambda )}a_i(h,e)D_{v_i}(h)+\delta _{\Phi(\lambda),p-2}\sum_{j=1}^{2}b_j(h,e)d_j(h)\\
			&=\sum_{i=0}^{ \Phi(\lambda )}a_i(h,e)(\lambda-2i)v_i,\\
			\varphi_1(e)=D_{h,e}(e)&=\sum_{i=0}^{ \Phi(\lambda )}a_i(h,e)D_{v_i}(e)+\delta _{\Phi(\lambda),p-2}\sum_{j=1}^{2}b_j(h,e)d_j(e)\\
			&=\sum_{i=0}^{ \Phi(\lambda )}a_i(h,e)i(\lambda-i+1)v_{i-1}+\delta _{\Phi(\lambda),p-2}b_1(h,e)v_{p-2}.
\end{align*}
Then $a_i(h,e)=0$ for $1\leq i\leq \Phi(\lambda )$ and $\delta _{\Phi(\lambda),p-2}b_1(h,e)=0$, since $\varphi_1(e)=0$.
As a result,
		\begin{equation}\label{1922}
			\varphi_1(h)=a_0(h,e)\lambda v_0.
		\end{equation}
On the other hand,
\begin{align*}
			\varphi_1(h)=D_{h,f}(h)&=\sum_{i=0}^{ \Phi(\lambda )}a_i(h,f)(\lambda-2i)v_i,\\
			\varphi_1(f)=D_{h,f}(f)&=\sum_{i=0}^{ \Phi(\lambda )}a_i(h,f)D_{v_i}(f)+\delta _{\Phi(\lambda),p-2}\sum_{j=1}^{2}b_j(h,f)d_j(f)\\
			&=\sum_{i=0}^{ \Phi(\lambda )}a_i(h,f)v_{i+1}+\delta _{\Phi(\lambda),p-2}b_2(h,f)v_0.
		\end{align*}
Then $a_i(h,f)=0$ for $0\leq i\leq \Phi(\lambda )-1$ and $\delta_{\Phi(\lambda), p-2}b_2(h,f)=0$, since $\varphi_1(f)=0$.
As a result,
		\begin{equation}\label{1923}
			\varphi_1(h)=-\lambda a _{\Phi(\lambda)}(h,f)v_{\Phi(\lambda )}.
		\end{equation}
From Equations (\ref{1922}) and (\ref{1923}), it follows that
		\begin{align*}
			a_0(h,e)\lambda v_0=-\lambda a _{\Phi(\lambda)}(h,f)v_{\Phi(\lambda )}.
		\end{align*}
If $\Phi(\lambda )=0$, then by Equation (\ref{1922}) or (\ref{1923}),
we get $\varphi_1(h)=0$.
If $\Phi(\lambda ) \neq0$, then $a_0(h,e)=a _{\Phi(\lambda)}(h,f)=0$,
which implies $\varphi_1(h)=0$.
Consequently, we prove Claim 1.

Let $x=x_ee+x_hh+x_ff \in \mathfrak{sl}(2)$ with $x_e,x_h,x_f\in \mathbb{F}$.
On one hand,
		\begin{align*}
			\varphi_1(e)=D_{x,e}(e)=\sum_{i=0}^{\Phi(\lambda)}a_i(x,e)i(\lambda+1-i)v_{i-1}+\delta _{\Phi(\lambda),p-2}b_1(x,e)v_{p-2}=0.
		\end{align*}
It follows that  $a_i(x,e)=0$ for $1\leq i \leq \Phi(\lambda)$ and $\delta _{\Phi(\lambda), p-2}b_1(x,e)=0$.
Then
		\begin{equation}\label{1435}
			\begin{aligned}
			\varphi_1(x)&=a_0(x,e)D_{v_0}(x)+\delta _{\Phi(\lambda),p-2}b_2(x,e)d_2(x)\\
			&=a_0(x,e)D_{v_0}(x_hh+x_ff)+\delta _{\Phi(\lambda),p-2}b_2(x,e)d_2(x_hh+x_ff)\\
			&=a_0(x,e)(x_h\lambda v_0+x_fv_1)+\delta _{\Phi(\lambda),p-2}b_2(x,e)x_fv_0.
			\end{aligned}
		\end{equation}
On the other hand,
\begin{align*}
			\varphi_1(f)=D_{x,f}(f)=\sum_{i=0}^{\Phi(\lambda)}a_i(x,f)v_{i+1}+\delta _{\Phi(\lambda),p-2}b_2(x,f)v_{0}=0.
		\end{align*}
It follows that $a_i(x,f)=0$ for $0\leq i \leq \Phi(\lambda)-1$  and $\delta _{\Phi(\lambda),p-2}b_2(x,f)=0$.
Then by definition, we get
		\begin{equation}\label{1436}
			\begin{aligned}
			\varphi_1(x)&=a_{\Phi(\lambda)}(x,f)D_{v_{\Phi(\lambda)}}(x)+\delta _{\Phi(\lambda),p-2}b_1(x,f)d_1(x)\\
			&=a_{\Phi(\lambda)}(x,f)D_{v_{\Phi(\lambda)}}(x_ee+x_hh)+\delta _{\Phi(\lambda),p-2}b_1(x,f)d_1(x_ee+x_hh)\\
			&=a_{\Phi(\lambda)}(x,f)(x_e\lambda v_{\Phi(\lambda)-1}-x_h\lambda v_{\Phi(\lambda)})+\delta _{\Phi(\lambda),p-2}b_1(x,f)x_ev_{p-2}.
			\end{aligned}
		\end{equation}
From Equations (\ref{1435}) and (\ref{1436}), we get
	  \begin{equation}\label{935}
	  		\begin{aligned}
		&a_0(x,e)(\lambda x_hv_0+x_fv_1)+\delta _{\Phi(\lambda),p-2}b_2(x,e)x_fv_0\\
		&=a_{\Phi(\lambda)}(x,f)(\lambda x_ev_{\Phi(\lambda)-1}-\lambda x_hv_{\Phi(\lambda)})+\delta _{\Phi(\lambda),p-2}b_1(x,f)x_ev_{p-2}.
			\end{aligned}
	  \end{equation}

\textit{Step 3.1: Let $\Phi(\lambda)=1$ and $p=3$.}
Then by Equation (\ref{935}), we have
\begin{align*}
(a_0(x,e)x_h+b_2(x,e)x_f)v_0+a_0(x,e)x_fv_1
=a_1(x,f)x_ev_0+(b_1(x,f)x_e-a_1(x,f)x_h)v_1.
\end{align*}
On one hand,
		\begin{align*}
			\varphi_1(h)=D_{x,h}(h)&=\sum_{i=0}^{1}a_i(x,h)D_{v_i}(h)+\sum_{j=1}^{2}b_j(x,h)d_j(h)\\
			&=\sum_{i=0}^{1}a_i(x,h)(1-2i)v_i=0
		\end{align*}
Then $a_0(x,h)=a_1(x,h)=0$, since $\varphi_1(h)=0$.
It follows that
\begin{equation}\label{2206}
			\begin{aligned}
				\varphi_1(x)&=D_{x,h}(x)=\sum_{j=1}^{2}b_j(x,h)d_j(x)\\
				&=b_1(x,h)(x_ev_1)+b_2(x,h)(x_fv_0).
			\end{aligned}
		\end{equation}
By Equations (\ref{1435}), (\ref{1436}) and (\ref{2206}), we get
\begin{align*}
			\varphi_1(x)&=a_1(x,f)x_ev_0+b_1(x,h)x_ev_1\\
			&=a_0(x,e)x_fv_1+b_2(x,h)x_fv_0.
		\end{align*}
As a result,
$$\varphi_1(x)=\left\{\begin{matrix}
   		0,&x_ex_f=0\\
   		m(x)v_0+n(x)v_1,& x_ex_f\neq 0,
   	\end{matrix}\right.$$
   where $$m(x)=b_2(x,h)x_f=a_1(x,f)x_e,\quad n(x)=a_0(x,e)x_f=b_1(x,h)x_e.$$

 \textit{Step 3.2: Let $2<\Phi(\lambda)\neq p-2$ or $\Phi(\lambda)=0$.}
  Then by Equation (\ref{935}), we get $\varphi_1(x)=0$.

 \textit{Step 3.3: Let $\Phi(\lambda)=2$ and $x_h\neq0$.}
  Then  by Equation (\ref{935}), we get $a_{\Phi(\lambda)}(x,f)=0$,
which implies $\varphi_1(x)=0$ by Equation (\ref{1436}).

 \textit{Step 3.4: Let $\Phi(\lambda)=2$ and $x_h=0$.} Then $\Phi(\lambda)\neq p-2$, and by definition   we get
		\begin{align*}
			\varphi_1(h)=D_{x,h}(h)=\sum_{i=0}^{2}a_i(x,h)(2-2i)v_i=0,
		\end{align*}
and then $a_i(x,h)=0$ for $i\neq1$.
Hence by definition, we have
$$\varphi_1(x)=a_1(x,h)(x_ee+x_ff)v_1=a_1(x,h)(2x_ev_0+x_fv_2).$$
It follows from Equations (\ref{1435}) and (\ref{1436}) that
		\begin{align*}
			a_0(x,e)x_fv_1=2a_2(x,f)x_ev_1=a_1(x,h)(2x_ev_0+x_fv_2).
		\end{align*}
Then  $a_1(x,h)=0$, that is, $\varphi_1(x)=0$.

 \textit{Step 3.5: Let $\Phi(\lambda)=1$ and $p\neq 3$.}
If $x_h=0$,
then by Equation (\ref{935}) we have $a_0(x,e)=a_1(x,f)=0$,
which implies $\varphi_1(x)=0$.
If $x_h\neq0$, then
		\begin{align*}
			\varphi_1(h)=D_{x,h}(h)=\sum_{i=0}^{1}a_i(x,h)D_{v_i}(h)=\sum_{i=0}^{1}a_i(x,h)(1-2i)v_i=0.
		\end{align*}
Since $1-2i\neq 0$, we get $a_i(x,h)=0,i=0,1$, and then $ \varphi_1(x)=0 $.

 \textit{Step 3.6: Let $\Phi(\lambda)=p-2$ and $p\neq 3$.}
  Then by Equation (\ref{935}), we get
		\begin{align*}
			&a_0(x,e)(x_h(p-2)v_0+x_fv_1)+b_2(x,e)x_fv_0\\
			=&a_{p-2}(x,f)(x_e(p-2)v_{p-3}+(2-p)x_hv_{p-2})+b_1(x,f)x_ev_{p-2}.
		\end{align*}
 As a result, $a_0(x,e)=a_{p-2}(x,f)=0$ since $p\neq3$,
that is, $\varphi_1(x)=0$.	
\end{proof}

\end{document}